\documentclass[a4paper, preprint, sort&compress, 3p]{elsarticle}

\usepackage[utf8]{inputenc}					

\usepackage{lmodern}						
\usepackage[english]{babel}					

\usepackage[colorlinks=true]{hyperref}
\usepackage{tikz}
\usetikzlibrary{calc, svg.path, patterns}

\usepackage{subcaption}
\usepackage{todonotes}

\DeclareMathAlphabet{\mathbbold}{U}{bbold}{m}{n}	
\newcommand{\idty}{\mathbbold{1}}

\usepackage{amsmath,amsfonts,amssymb,amsthm}

\theoremstyle{plain}
\newtheorem{theorem}{Theorem}
\newtheorem{prop}{Proposition}
\newtheorem{corollary}{Corollary}
\newtheorem{lemma}{Lemma}
\theoremstyle{definition}
\newtheorem{definition}{Definition}
\theoremstyle{remark}
\newtheorem{rmk}{Remark}

\DeclareMathOperator{\diag}{diag}	
\DeclareMathOperator{\sampl}{sampl}

\newcommand{\bR}{\mathbb{R}}					
\newcommand{\bC}{\mathbb{C}}					
\newcommand{\bG}{\mathbb{G}}
\newcommand{\bK}{\mathbb{K}}
\newcommand{\bH}{\mathbb{H}}
\newcommand{\bS}{\mathbb{S}}

\newcommand{\bN}{\mathbb{N}}					
\newcommand{\bZ}{\mathbb{Z}}
\newcommand{\cC}{\mathcal{C}}
\newcommand{\cF}{\mathcal{F}}

\newcommand{\cO}{\mathcal{O}}
\newcommand{\cP}{\mathcal{P}}
\newcommand{\cS}{\mathcal{S}}

\newcommand{\cJ}{\mathcal{J}}	
\newcommand{\cU}{\mathcal{U}}	

\newcommand{\shift}{S}

\DeclareMathOperator{\AP}{AP}

\DeclareMathOperator{\ev}{ev}

\title{Generalized Fourier-Bessel operator and almost-periodic interpolation and approximation}

\author[jp]{J.-P.\ Gauthier}
\ead{gauthier@univ-tln.fr}
\address[jp]{LSIS, Université de Toulon, La Garde CEDEX, France. \emph{Email:}
  \texttt{gauthier@univ-tln.fr}}

\author[dp]{D.\ Prandi}
\ead{dario.prandi@l2s.centralesupelec.fr}
\address[dp]{L2S, CentraleSupelec, Gif-sur-Yvette, France. \emph{Email:} \texttt{dario.prandi@l2s.centralesupelec.fr}}

\begin{document}

\begin{abstract}
We consider functions $f$ of two real variables, given as trigonometric functions over a finite set $F$ of frequencies.
This set is assumed to be closed under rotations in the frequency plane of angle $\frac{2k\pi}{M}$ for some integer $M$. 
Firstly, we address the problem of evaluating these functions over a similar finite set $E$ in the space plane and, secondly, we address the problems of interpolating or approximating a function $g$ of two variables by such an $f$ over the grid $E.$
In particular, for this aim, we establish an abstract factorization theorem for the evaluation function, which is a key point for an efficient numerical solution to these problems. This result is based on the very special structure of the group $SE(2,N)$, subgroup of the group $SE(2)$ of motions of the plane corresponding to discrete rotations, which is a maximally almost periodic group.

Although the motivation of this paper comes from our previous works on biomimetic image reconstruction and pattern recognition, where these questions appear naturally, this topic is related with several classical problems: the FFT in polar coordinates, the Non Uniform FFT, the evaluation of general trigonometric polynomials, and so on.

\end{abstract}

\maketitle

\section{Introduction}

In several areas of physics and biomedical engineering (e.g., for NMRs), data in the frequency domain is obtained on polar grids, i.e., grids rotationally invariant under a certain number of discrete rotations.
The extraction of spatial information from this data thus requires the implementation of an efficient inverse Fourier Transform on such grids.
This problem is the focus of a large literature spanning a lot of different techniques as, for example, the implementation of a Polar FFT, as in \cite{bbbb, dddd, iiii}, or the application of the general theory of non-equispaced FFT, as in \cite{hhhh,cccc,gggg,kamm,aaa}.

A different, very natural but numerically challenging, possible solution is the exploitation of the decomposition of the 2D Fourier transform in polar coordinates. Indeed, it is well-known that the Fourier Transform $\hat f$ of a square integrable function $f$ can be obtained by, firstly, developing $f$ in a multipole series $f(\rho e^{i\theta}) = \sum_{n\in\bZ} f_n(\rho) e^{in\theta}$,  and then applying the Hankel transform on the $f_n$'s. More precisely, see Section~\ref{sec:se2} for details, polar coordinates allow to identify $L^2(\bR^2)\simeq L^2(\bS^1)\otimes L^2(\bR_+,\rho\,d\rho)$. Then, letting $\cF:L^2(\bS^1)\to L^2(\bZ)$ be the Fourier transform on $\bS^1$ and $\cF_{\bR^2}:L^2(\bR^2)\to L^2(\bR^2)$ be the one on $\bR^2$, we have
\begin{equation}\label{eq:hankel}
	\cF_{\bR^2} = (\cF^*\otimes \idty)\circ \cJ \circ (\cF\otimes \idty).
\end{equation}
Here, we implicitly identified $L^2(\bZ)\otimes L^2(\bR_+,\rho\,d\rho)\simeq \bigoplus_{k\in \bZ} L^2(\bR_+,\rho\,d\rho)$, and let $\cJ=  \bigoplus_{n\in \bZ} \cJ_{n}$ be the so-called \emph{Fourier-Bessel operator}. Namely, $\cJ_n$ is a renormalized version of the $n$-th Hankel transform operator:
\begin{equation}
	\cJ_n \varphi(\lambda) = (-i)^n\int_{\bR_+} \varphi(\rho)J_n(\lambda\rho)\,\rho d\rho, \qquad \varphi:\bR_+\to\bR,
\end{equation}
where $J_n$ is the $n$-th Bessel function of the first kind.

Clearly, due to the complicated (non-periodic) nature of Bessel functions, the main difficulty encountered with this approach is the computation of the Hankel transform, for which many algorithms have been proposed, see e.g., \cite{cerjan, knock} or \cite{key} for a review. 
In this paper, we propose a completely different approach from what is usually taken: Instead of approximating $f$ in the continuous space variable by a discrete sampling, we develop a method to exactly evaluate a counterpart of \eqref{eq:hankel} for discretized frequencies. 

We propose to consider the data to represent not square integrable functions of $L^2(\bR^2)$, but a finite-dimensional subspace of \emph{Bohr almost-periodic functions}. 
That is, we assume to be given a set of frequencies $\tilde F\subset\bR^2$ invariant under rotations of multiples of $\tfrac{2\pi}N$, $N\in\bN$,  representing the frequencies obtained via the sensors (see Figure~\ref{fig:bispectral}), and we focus on functions $f:\bR^2\to \bC$ with frequencies in $\tilde F$. That is,   
\begin{equation}\label{eq:almost-per-func}
	f(x) = \sum_{\lambda\in \tilde F} \hat f(\lambda) e^{i\langle \lambda, x\rangle}, \qquad \forall x\in\mathbb R^2, 
	\quad\text{where}\quad \hat f:\tilde F\to \bC.
\end{equation}
Observe that, if the set $\tilde F$ is contained in a periodic square lattice of $\bR^2$, the above essentially coincides with the discrete Fourier transform, see \cite{cool,at, ATW}.

Let $\cS = \{ \rho e^{i\theta} \in\bR^2\mid \theta\in[0,2\pi/N) \}\subset \bR^2$ be the slice of angle $2\pi/N$.
Then, by its rotational invariance, we have that $\tilde F$ is completely determined by some $F\in\cS$. Namely, letting $R_\theta$ be the rotation of $\theta$ around the origin, we have $\tilde F = \bigcup_{k=0}^{N-1}R_{\frac{2\pi}N k}(F)$.
The space of functions of the form \eqref{eq:almost-per-func}, denoted by $\AP_F(\bR^2)$, is then invariant under the action of the semidiscrete group of rototranslations $SE(2,N)$, that is, the subgroup of the group $SE(2)$ of Euclidean motions where only the rotations $\{R_{\frac{2\pi}N k}\}_{k=0}^{N-1}$ are allowed.

As a first step in our analysis, exploiting the deep connection between decomposition \eqref{eq:hankel} and the action of the group
of rototranslations $SE(2) = \bS^1\ltimes\bR^2$ on $L^2(\bR^2)$, we generalize \eqref{eq:hankel} to almost-periodic function on a group $\bG$, where $\bG = \bK\ltimes \bH$ is the semidirect product of two abelian groups, with $\bK$ finite. 
The finiteness of $\bK$ then allows us to express decomposition \eqref{eq:hankel} in terms of a generalized Fourier-Bessel operator, which turns out to be well defined on a large set of numerically relevant functions: finite-dimensional $\bK$-invariants sets of Bohr almost periodic functions on $\bG$.
These sets, denoted by $\AP_F(\bG)$, are composed of linear combinations of the matrix coefficients of a finite set $F$ of irreducible unitary $\#(\bK)$-dimensional representations of $\bG$. 
This is the content of Theorem~\ref{thm:general} in Section~\ref{sec:harmonic}.
Then, in the remainder of the section, we consider the problem of determining an almost-periodic function in $\AP_F(\bG)$ which interpolates a given function $\psi:\bG\to \bC$ on some finite set $\tilde E\subset \bG$, invariant under the action of $\bK$. Indeed, in Theorem~\ref{thm:discrete}, we present a discretized version of decomposition \eqref{eq:hankel}, in terms of a discretization of the generalized Fourier-Bessel operator. 

In the second part of the paper, setting $\bG$ to be the semidiscrete group of rototranslations $SE(2,N)$, we particularize the theoretical results of the first part and present numerical algorithms for the (exact) evaluation, interpolation, and approximation of functions of the form \eqref{eq:almost-per-func} on finite sets of spatial samples $\tilde E\subset \bR^2$, invariant under discrete rotations of $2\pi k/N$, $k\in\bZ_N$. This is an instance of a very general problem, and can be seen as a generalization of the discrete Fourier Transform and its inverse, that act on regular square grids, i.e., invariant under the the action of $SE(2,4)$.

\begin{figure}
	\centering
	\includegraphics[width=.4\textwidth]{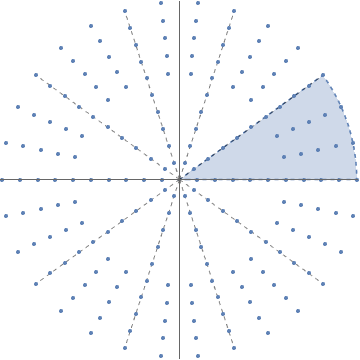}
	\caption{An example of the considered sets of frequencies/spatial samples, for $N=10$. The blue region is $\mathcal S_N$, see Section~\ref{sec:ap-interp}.}
	\label{fig:bispectral}
\end{figure}

To better precise the above result, let $Q,P\in\bN$ be such that $\#(F) =N\times Q$ and $\#(E) = N\times P$, and denote by $\bC^{\tilde F}\simeq \bC^N\otimes\bC^Q$ and $\bC^{\tilde E}\simeq \bC^N\otimes\bC^P$ the sets of complex-valued functions over $\tilde F$ and $\tilde E$, respectively. 
Additionally, we denote by $\sampl(\varphi) \in \bC^{\tilde E}$ the sampling of a function $\varphi:\bR^2\to \bC$ on $\tilde E$. 
Then, the \emph{evaluation operator}, denoted by $\ev$, associates to the coefficients $\hat f$ the corresponding samples $\sampl(f)$, according to \eqref{eq:almost-per-func}. That is,
\begin{equation}
	\begin{array}{cccc}
		\ev & \bC^{\tilde F} & \to & \bC^{\tilde E} \\
		 & \hat f & \mapsto & \sampl(f).
	\end{array}	
\end{equation}
It is clear that $\ev$ depends only on the choice of $\tilde E$ and $\tilde F$. 

From a numerical perspective, a naive implementation of $\ev$ requires the evaluation of a $PN\times QN$ complex dense matrix $A$ and its multiplication with a vector of length $QN$. Even assuming $A$ to be precomputed, in the best case scenario where $Q = P$, via the Coppersmith-Winograd algorithm this yields a computational complexity of $\mathcal O(Q^{2.376}N^{2.376})$. Up to a prefactorization of $A$, e.g.\ via a QR decomposition, the computation of $\ev^{-1}$ has then a computational complexity of $\mathcal O(Q^2N^2)$.
Our main (practical) result, contained in Corollary~\ref{cor:ap-r} (see Section~\ref{sec:image}), is that the row-wise discrete Fourier transform interwines $\ev$ with a block diagonal operator whose blocks have size $P\times Q$: the discrete Fourier-Bessel operator. This allows to lower the computation complexity of computing $\ev$ to $\mathcal O(QN\log N+NQ^{2.376})$ and $\ev^{-1}$ to $\mathcal O(N(Q^2+Q\log N))$.

As a final note, we point out that recent contributions of the authors on the mathematical structure of the primary visual cortex, have shown how the exploitation the action of the semi-discrete group of rototranslations can yield an efficient and natural (biomimetic) framework both for image reconstruction and for pattern recognition \cite{bosc1, bosc, gau, bohi}.

\section{Connection between Bessel functions of the first kind and the group of Euclidean motions}\label{sec:se2}

For more details on what follows, we refer the reader to \cite[Ch. 4]{vilenkin}. 
Recall that the group of Euclidean motions $SE(2)$ is the semidirect product $\bS^1\ltimes\bR^2$, under the action of rotations $R:\bS^1\to\cU(\bR^2)$. In the following we still denote by $R$ the contragredient action of rotations on $L^2(\bR^2)$. Namely, $R_\theta f(x) = f(R_{-\theta}x)$ for all $f\in L^2(\bR^2)$.

The dual object $SE(2){\hat{}}$ of $SE(2)$, i.e.\ the set of equivalence classes of unitary irreducible representations, can be completely determined by Mackey machinery. Indeed, it turns out that $SE(2){\hat{}}$ can be parametrized by $\bZ\sqcup \bR_+$ as follows: To $\lambda>0$ it corresponds the following unitary irreducible representation on $L^2(\bS^1)$:
\begin{equation}
	\chi^\lambda(\theta,x) = \diag_\alpha\left(e^{i \langle R_\alpha\Lambda ,  x\rangle}\right) S(\theta), 
	\qquad\text{with } \Lambda=(\lambda, 0).
\end{equation}
Here, we denoted by $\diag_\alpha \varphi_\alpha$ the multiplication operator by $\varphi\in L^2(\bS^1)$ and by $S:\bS^1\to \cU(L^2(\bS^1))$ the shift operator $S(\theta)\varphi(\alpha) = \varphi(\alpha-\theta)$.
On the other hand, morally for $\lambda = 0$, to each $n\in\bZ$ corresponds the representation $\chi^{0,n}(\theta,x) = e^{i n\theta}$ on $\bC$.

\begin{prop}\label{prop:coeff-se2}
	The matrix elements of $\chi^\lambda$, $\lambda\in\bR_+$, with respect to the basis $\{e^{in\theta}\}_{n\in\bZ}$ of $L^2(\bS^1)$ are
	\begin{equation}
		c^\lambda_{m,n}(\theta,\rho e^{i\varphi}):= \frac1{2\pi}\int_0^{2\pi} \chi^\lambda(\theta,x) e^{in\alpha} \, e^{-im\alpha}\,d\alpha = e^{-in\theta}e^{i(n-m)\varphi}i^{n-m}J_{n-m}(\lambda \rho).
	\end{equation}
\end{prop}

\begin{proof}
	By definition of $\chi^\lambda$ we have
	\begin{equation}
		c^\lambda_{m,n}(\theta,\rho e^{i\varphi}) 
		= \frac1{2\pi}e^{-in\theta}\int_0^{2\pi} e^{i\lambda\rho\cos(\varphi-\alpha)} e^{i(n-m)\alpha} \,d\alpha
		= \frac1{2\pi}e^{-i n\theta}e^{i(n-m)\varphi}\int_0^{2\pi} e^{i(\lambda\rho\cos\gamma-(n-m)\gamma)} \,d\gamma.
	\end{equation}
	Let us observe that, by the integral characterization of $J_{n-m}$, the change of variables $\eta = \pi/2-\gamma$ yields
	\begin{equation}\label{eq:besselint}
		\int_0^{2\pi} e^{i(\lambda\rho\cos\gamma-(n-m)\gamma)} \, d\gamma
		= (-i)^{n-m}\int_0^{2\pi} e^{-i(\lambda\rho\sin\eta +(n-m)\eta)} \, d\alpha
		= 2\pi i^{n-m} J_{n-m}(\lambda\rho),
	\end{equation}
	thus completing the proof.
\end{proof}
	
\begin{definition}
	The \emph{Fourier-Bessel operator} $\cJ$ is defined on $\bigoplus_{n\in\bZ}L^2(\bR_+,\rho\,d\rho)\cap L^1(\bR_+,\sqrt{\rho}\,d\rho)$ by
	\begin{equation}
		\cJ = \bigoplus_{n\in\bZ} \cJ_{n},
	\end{equation}
	where $\cJ_{n}$, the \emph{Fourier-Bessel operator of order $n\in\bZ$}, is given by
	\begin{equation}
		\cJ_n\varphi(\lambda) = (-i)^n\int_{0}^{+\infty} \varphi(\rho)J_{n}(\rho)\,\rho d\rho
		\qquad
		\forall \varphi\in L^2(\bR_+,\rho\,d\rho)\cap L^1(\bR_+,\sqrt{\rho}\,d\rho).
	\end{equation}
\end{definition}
The fact that the Fourier-Bessel operator is well-defined under the above restrictions is a direct consequence of the asymptotic formulae for $J_n$. 

Recall that polar coordinates on $\bR^2$ naturally induce the isometry $\Xi:L^2(\bR^2)\to L^2(\bS^1)\otimes L^2(\bR_+,\rho\,d\rho)$. Letting $\cF_{\bS^1}:L^2(\bS^1)\to L^2(\bZ)$ be the Fourier transform on $\bS^1$, the operator $\cF_{\bS^1}\otimes\idty$ is an isometry between $L^2(\bS^1)\otimes L^2(\bR_+,\rho\,d\rho)$ and $L^2(\bZ)\otimes L^2(\bR_+,\rho\,d\rho)$. 
Moreover, we have an isometry $\cP$ between the latter space and $\bigoplus_{k\in \bZ}L^2(\bR_+,\rho\,d\rho)$. Considering the canonical basis $\{\delta_n\}_{n\in\bZ}$ on $L^2(\bZ)$, where $\delta_n$ is the Kroenecker delta, this isometry is defined on simple tensors by
\begin{equation}
	\cP\left(\sum_{n\in\bZ}a_n \delta_n\otimes \varphi \right) := (a_n\varphi)_{n\in\bN} \in \bigoplus_{n\in\bZ}L^2(\bR_+,\rho\,d\rho).
\end{equation}

The next result is exactly \eqref{eq:hankel}, although in the latter we implicitly assumed the identifications given by $\Xi$ and $\cP$.
In particular, it allows to connect the Fourier-Bessel operator with the 2D Fourier transform and gives a Plancherel Theorem for the Fourier-Bessel operator. 

\begin{theorem}\label{thm:hankel}
	Let $\cF_{\bR^2}$ be the Fourier transform on $\bR^2$.
	The Fourier-Bessel operator uniquely extends to an isometry of $\bigoplus_{n\in\bZ}L^2(\bR_+,\rho\,d\rho)$ such that
	\begin{equation}\label{eq:main-se}
		\Xi\circ\cF_{\bR^2}\circ\Xi^{-1} = (\cF_{\bS^1}^*\otimes\idty)\circ\cP^{-1} \circ \cJ \circ \cP \circ(\cF_{\bS^1}\otimes\idty).
	\end{equation}
\end{theorem}

\begin{proof}
	Let $n\in\bZ$, $\varphi\in L^2(\bR_+,\rho\,d\rho)\cap L^1(\bR_+,\sqrt{\rho}\,d\rho)$, and consider $f(\rho e^{i\theta}) = e^{in\theta}\varphi(\rho)$. We claim that \eqref{eq:main-se} holds for such $f$'s, that is,
	\begin{equation}\label{eq:2dF}
		\cF_{\bR^2}f(\lambda e^{i\omega}) = e^{in\omega} \cJ_{n}\varphi(\lambda).
	\end{equation}
%
	Indeed, a simple computation yields
	\begin{equation}
		\cF_{\bR^2} f(\lambda e^{i\omega}) 
		= \frac1{2\pi}\int_{0}^{+\infty} \varphi(\rho) \int_0^{2\pi}  e^{-i(\rho\lambda\cos(\omega-\theta)-n\theta)}\,d\theta\,\rho d\rho
		= \frac1{2\pi} e^{in\omega} \int_{0}^{+\infty} \varphi(\rho) \int_0^{2\pi}  e^{-i(\rho\lambda\cos(\gamma)+n\gamma)}\,d\theta\,\rho d\rho.
	\end{equation}
	Then, for such a $\varphi$, \eqref{eq:2dF} follows by the computations in \eqref{eq:besselint} and the fact that $J_n$ is real-valued. Indeed, these yield,
	\begin{equation}
		\cF_{\bR^2} f(\lambda e^{i\omega}) 
		= e^{in\omega} (-i)^n\int_{0}^{+\infty} \varphi(\rho)  J_n(\lambda\rho)\,\rho d\rho
		= e^{in\omega} \cJ_{n}\varphi(\lambda). 
	\end{equation}
	
	Since, functions $f$'s as above form a basis for $\bigoplus_{n\in\bZ}L^2(\bR_+,\rho\,d\rho)\cap L^1(\bR_+,\sqrt\rho\,d\rho)$, we have obtained,
	\begin{equation}
		\cJ = \cP\circ(\idty\otimes\cF_{\bS^1})\circ \Xi\circ \cF_{\bR^2}\circ\Xi^{-1}\circ (\idty\otimes\cF_{\bS^1}^*)\circ\cP^{-1} 
		\quad \text{on}\quad
		\bigoplus_{n\in\bZ}L^2(\bR_+,\rho\,d\rho)\cap L^1(\bR_+,\sqrt\rho\,d\rho).
	\end{equation}
	The result then follows from the Plancherel Theorem for $\cF_{\bR^2}$, which implies that the r.h.s.\ is an isometry, and the fact that $L^2(\bR_+,\rho\,d\rho)\cap L^1(\bR_+,\sqrt\rho\,d\rho)$ is dense in $L^2(\bR_+,\rho\,d\rho)$.
\end{proof}

%

\section{The generalized Fourier-Bessel operator and its application to almost-periodic interpolation}\label{sec:harmonic}

In this section we consider the general setting of a semidirect product group $\bG = \bK \ltimes \bH$, where $\bH$ is an abelian locally compact group and $\bK$ is an abelian finite group of cardinality $N$ acting on $\bH$ via $\phi:\bK\to \operatorname{Aut}(\bH)$. 
The Haar measure on $\bK$ is the discrete one, so that $L^2(\bK) = \bC^\bK\simeq\bC^N$.
We denote by $\shift:\bK\to\mathcal U(\bC^\bK)$ the left-regular representation of $\bK$. That is, 
\begin{equation}
	\shift(k).v(h) = v(k^{-1}h),\qquad \forall k,h\in\bK,\, v\in \bC^\bK.
\end{equation}
The concrete example to bear in mind of such $\bG$ is the semidiscrete group of rototranslations $SE(2,N)$, where $\bH=\bR^2$, $\bK=\bZ_N$, and $\phi(k) = R_{\frac{2\pi k}N}$. 

The action $\phi$ defines a contragredient action on the Pontryagin's dual $\phi:\bK\to\operatorname{Aut}({\widehat\bH})$, whose action on $\lambda\in\widehat \bH$ is
\begin{equation}
	\phi(k)\lambda(x) := \lambda(\phi(k^{-1}) x), \qquad \forall k\in\bK,\, x\in\bH.
\end{equation}
The set of orbits of $\widehat\bH$ under the action of $\bK$ with trivial stabilizer subgroup is denoted by $\widehat\cS\subset \widehat\bH$. 
As shown in Lemma~\ref{lem:irrep} in \ref{sec:aux}, to each $\lambda\in\widehat\cS$ we can associate the irreducible representation $T^\lambda$, acting on $\bC^\bK$, given by
\begin{equation}\label{eq:repr-semidi}
	T^\lambda(k,x) =  \diag_h (\phi(h)\lambda(x))\, S(k).
\end{equation}
Here, $\diag_h (v_h)$ is the diagonal operator corresponding to a given vector $v\in \bC^\bK$. We remark that, under some mild assumptions on $\bG$, by Mackey machinery \cite{Barut77theoryof} these are the only $N$-dimensional irreducible representations of $\bG$. When, as it happens for $SE(2,N)$, all $\lambda\in\widehat \bH\setminus\{0\}$ have trivial stabilizer subgroup, to  get all irreducible representations of $\bG$ it suffices to add to the $T^\lambda$'s the one-dimensional representations of $\bK$.

\subsection{Almost-periodic functions}
The Bohr compactification (see \cite{dixmier,weyl}) of $\bG$ is the universal object $(\bG^\flat,\Pi)$ in the category of diagrams
	$\pi:\bG \to \mathbb B$, 
where $\pi$ is a continuous homomorphism from $\bG$ to a compact group $\mathbb B$.  Then, all continuous maps of $\bG$ in a compact group factor out w.r.t.\ the inclusion $\Pi: \bG \hookrightarrow \bG^{\flat}$. That is, if $\mathbb B$ is a compact group and $\Psi:\bG\rightarrow \mathbb B$ is continuous, then
\[
\Psi=\psi\circ\Pi,\text{ for a continous map }\psi%
:\bG^{\flat}\rightarrow \mathbb B\text{ .}%
\]
By \cite[16.5.3]{dixmier}, the group $\bG$ is always maximally almost periodic, that is, the inclusion $\Pi$ is a continuous \emph{injective} homomorphism.
Regarding duality over this kind of groups, the reader is advised to consult
\cite{heyer}, see also \cite{hew} for compact groups.
Then, the set $\AP(\bG)$ of almost periodic functions in the sense of Bohr over $\bG$ is the set of continuous functions $f:\bG\to\mathbb C$ that lift to continuous functions over $\bG^{\flat}$. That is, $f=\Psi\circ\Pi$ for a certain continuous $\Psi:\bG^{\flat}\to\mathbb{C}$.

It is classical that almost-periodic functions are the uniform limit of linear combinations of coefficients of finite-dimensional unitary representations \cite[Theorem 16.2.1]{dixmier}. 
A crucial point is the following: Any $\lambda\in\widehat\cS$ determines an $N$-dimensional space of almost-periodic functions which is invariant under the action of the representation $T^{\lambda}$ (which makes sense on $\AP(\bG)$, and is obviously unitarily equivalent to the $T^{\lambda}$ defined above). 
The rest of this section is indeed devoted to the study of the subspaces of almost-periodic functions determined by a finite set of representations of the form \eqref{eq:repr-semidi}. 
That is, for a given finite set $F\subset \widehat\cS$, we focus on functions $f\in\AP(\bG)$ such that
\begin{equation}\label{eq:abstract-ap}
	f(g) = \sum_{\lambda\in F} \sum_{n,m\in\bK}\langle T^\lambda(g) \delta_n, \delta_m\rangle \,a_{n,m}(\lambda), \qquad a_{n,m}:F\to\bC.
\end{equation}
Here, we let $\{\delta_n\}_{n\in\bK}$ be the canonical basis of $L^2(\bK)\simeq \bC^\bK$. We denote by $\AP_F(\bG)$ the set of almost-periodic functions of the form \eqref{eq:abstract-ap}.
Direct computations yield the following.

\begin{prop}\label{prop:ap-simpl}
	For any $f\in\AP_F(\bG)$ and any $(k,x)\in\bG$, it holds,
	\begin{equation}
		f(k,x) = \sum_{\lambda\in F} \sum_{n\in\bK} \phi(nk)\lambda(x) \,\hat f(k, n, \lambda),\qquad\text{where } \hat f(k,n,\lambda) = a_{n,nk}(\lambda).
	\end{equation}
In particular, there exists a (linear) bijection $\cF_{\AP}:\AP_F(\bG)\to \bC^\bK\otimes\bC^\bK\otimes \bC^F$, mapping $f$ to $\hat f$.
\end{prop}

Observe that, since $\bG^\flat = \bK\ltimes\bH^\flat$, and $\bK$ is finite, it holds that $f\in\AP(\bG)$ if and only if $f(k,\cdot)$ is an almost-periodic function over the abelian group $\bH$, for all $k\in\bK$.
For this reason, there exists a natural embedding of $\AP(\bH)$, the set almost-periodic functions over $\bH$, in $\AP(\bG)$ that acts by lifting $\psi\in\AP(\bH)$ to $\Psi\in\AP(\bG)$ given by $\Psi(k,x) = \delta_0(k)\,\psi(x)$. 
As an immediate consequence of this fact and of Proposition~\ref{prop:ap-simpl}, we get the following.

\begin{corollary}\label{cor:ap-simpl-H}
	Let $\AP_F(\bH)$ be the subspace of $\AP(\bH)$ of almost-periodic functions on $\bH$ that lift to $\AP_F(\bG)$. Then, for any $\psi\in\AP_F(\bH)$ we have $a_{n,m}=0$ if $n\neq m$. Moreover, for any $x\in\bH$ it holds,
	\begin{equation}
		\psi(x) = \sum_{\lambda\in F} \sum_{n\in\bK} \phi(n)\lambda(x) \,\hat \psi(n, \lambda),\qquad\text{where } \hat \psi(n,\lambda) = a_{n,n}(\lambda).
	\end{equation}
	In particular, there exists a (linear) bijection $\cF_{\AP}:\AP_F(\bG)\to \bC^\bK\otimes \bC^F$, mapping $\psi$ to $\hat\psi$.
\end{corollary}

Later on, in Corollary~\ref{cor:general-H}, we will exploit the above in order to derive a decomposition of $\cF_{\AP}^{-1}$ over finite-dimensional subspaces of $\AP(\bH)$ from the corresponding result on finite-dimensional subspaces of $\AP(\bG)$.

\begin{rmk}
	The above lift from $\AP_F(\bH)$ to continuous functions in $\bG$ works exactly because $\bK$ is discrete. In order to perform this operation with a non-discrete $\bK$, as it happens in the case of the Euclidean group of motions $SE(2)$, one should fix a (continuous) function $\eta:\bK\to \bC$ and lift $\psi$ to $\Psi(k,x)=\eta(k)\psi(x)$. In this case a modified version of the above result still holds. However, Corollary~\ref{cor:general-H} does not. (See Remark~\ref{rmk:4}.)
\end{rmk}

\begin{rmk}\label{rmk:2}
	The above discussion is for the most part independent of the finiteness of $\bK$, and one could be tempted to apply these techniques to the case $\bG=SE(2)$.
	However, since the only finite-dimensional representations of $SE(2)$ are the characters of $\bS^1$, almost-periodic functions on the former are uniform limits of linear combinations of $(\theta,x)\mapsto e^{in\theta}$ and, as such, independent of $x\in\bR^2$.
	This shows that almost-periodic functions on $\bR^2$ \emph{cannot} be lifted to almost-periodic functions on $SE(2)$.
	This is a reflection of the fact that $SE(2)$, contrarily to $SE(2,N)$ is not maximally almost-periodic\footnote{One has to pay attention here: Maximal almost-periodic groups can have infinite-dimensional representations, but $SE(2)$ has ``not enough'' finite-dimensional ones to be maximal almost-periodic.}, and is one of the main reasons why in the following we will consider only the action of $SE(2,N)$, and not of $SE(2)$, on images.
\end{rmk}


\subsection{Generalized Fourier-Bessel operator}

Recall that the matrix coefficients of $T^\lambda$, with respect to the basis $\widehat\bK$ of $L^2(\bK)\simeq\bC^\bK$, are the functions,
\begin{equation}
	t^\lambda_{\hat m, \hat n}(g) := \sum_{\ell\in\bK} T^\lambda(g)\hat n(\ell)\,\overline{\hat m(\ell)},\qquad \forall g\in \bG, \, \hat n, \hat m\in \widehat\bK.
\end{equation}
Since, in the case of $SE(2)$, Bessel functions appeared inside in these coefficients, we now compute them in order to obtain a coherent generalization of Bessel functions to this context. 

In order to do so, let us mimic the polar coordinates construction, by choosing a bijection of $\bH/\bK\times\bK$ to $\bH$. To this aim, fix any section $\sigma: \bH/\bK\to \bH$, that we do not assume to have any regularity. Indeed, the arguments that follow work even for non-measurable $\sigma$'s. Then, $\Xi:f\mapsto f\circ \sigma$ is a bijection between functions on $\bG$ and functions on $\bK\times\bK\times\bH/\bK$. More precisely, if $f:\bG\to \bC$, then $\Xi f(k,h,y) := f(k,\phi(h)\sigma(y))$. 

\begin{prop}
	The matrix elements of $T^\lambda$, $\lambda\in \widehat\cS$, with respect to the basis $\widehat \bK$ of $L^2(\bK)$ are
	\begin{equation}
		\Xi\, t^\lambda_{\hat m, \hat n} (k,h,y) = 
		\overline{\hat n}(k)\, [\hat n - \hat m](h) \sum_{\ell\in\bK}\lambda(\phi(\ell)\sigma(y))\overline{[\hat n-\hat m]}(\ell)
	\end{equation}
\end{prop}

\begin{proof}
	It suffices to adapt the computations of Proposition~\ref{prop:coeff-se2} to this context. Namely, we have,
	\begin{equation}
		\begin{split}
			t^\lambda_{\hat m, \hat n}(k,\phi(h)\sigma(y)) 
			&= \sum_{\ell\in\bK} T^\lambda(k,\phi(h)\sigma(y))\hat n(\ell)\,\overline{\hat m(\ell)}\\
			&= \sum_{\ell\in\bK} \lambda(\phi(\ell^{-1}h)\sigma(y))\hat n(k^{-1}\ell)\,\overline{\hat m(\ell)}\\
			&= \overline{\hat n}(k) \sum_{\ell\in\bK} \lambda(\phi(\ell^{-1}h)\sigma(y))[\hat n-\hat m](\ell)\\
			&= \overline{\hat n}(k)[\hat n-\hat m](h) \sum_{r\in\bK} \lambda(\phi(r)\sigma(y))\overline{[\hat n-\hat m]}(r). \qedhere
		\end{split}
	\end{equation}
\end{proof}

The above proposition justifies the following.

\begin{definition}\label{def:gen-bessel}
	The \emph{generalized Bessel function of parameters $\hat n\in\widehat\bK$} is the function defined by
	\begin{equation}
		J_{\hat n}:\widehat\cS \times H/K\to \bC, \qquad J_{\hat n}(\lambda, y) = \sum_{r\in\bK} \lambda(\phi(r)\sigma(y))\overline{\hat n(r)}.
	\end{equation}
	The \emph{generalized Fourier-Bessel operator} is the operator 
	\begin{equation}
		\cJ:\bigoplus_{k\in\bK,\hat n\in \widehat \bK}\bC^F\to \bigoplus_{k\in\bK,\hat n\in \widehat \bK} C(\bH/\bK), \qquad \cJ = \bigoplus_{k\in\bK,\hat n\in \widehat \bK} \overline{\hat n(k)} \cJ_{\hat n},
	\end{equation}
	where $C(\bH/\bK)$ is the set of continuous functions on $\bH/\bK$, and $\cJ_{\hat n}$ is the operator with kernel $J_{\hat n}$. That is,
	\begin{equation}
		\cJ_{\hat n}\varphi(y) := \sum_{\lambda\in F} J_{\hat n}(\lambda, y)\varphi(\lambda),\qquad \forall \varphi\in \bC^F.
	\end{equation}
\end{definition}

\begin{rmk}
	The generalized Bessel functions depend on the choice of the section $\sigma$.
	Namely, if a different section $\sigma':\bH/\bK\to \bH$ is fixed we have that $\sigma'(y)=\phi(r)\sigma(y)$, and hence
	\begin{equation}
		J^{\sigma'}_{\hat n}(\lambda, y) = \hat n(r)\, J^{\sigma}_{\hat n}(\lambda, y).
	\end{equation}
\end{rmk}

Let $\cF:\bC^\bK\to \bC^{\widehat\bK}$ be the Fourier transform over $\bK$, defined for $v\in \bC^\bK$ as
\begin{equation}
	\cF v(\hat k) = \frac1{\sqrt N} \sum_{h\in\bK} \overline{\hat k(h)} v(h), \qquad \forall\hat k\in\widehat\bK.
\end{equation}
Moreover, for any vector space $V$ let $\cP_V: \bC^\bK \otimes \bC^{\widehat \bK}\otimes V \to \bigoplus_{k\in\bK,\hat n\in \widehat \bK} V$ be the bijection defined by 
\begin{equation}\label{eq:P}
	\cP_V(a\otimes b\otimes v) 
	= (a_k b_{\hat n} v)_{k\in\bK,\hat n\in\widehat\bK}, \qquad \forall a\in \bC^\bK,\, \, b\in\bC^{\widehat \bK},\, v\in V.
\end{equation}
We then have the following generalization of Theorem~\ref{thm:hankel}.

\begin{theorem}\label{thm:general}
	The bijection $\cF_{\AP}^{-1}:\bC^\bK\otimes\bC^\bK\otimes\bC^F\to \AP_F(\bG)$ admits the following decomposition
	\begin{equation}
		\cF_{\AP}^{-1} = \Xi^{-1}\circ(\idty \otimes\cF^*\otimes\idty)\circ \cP^{-1}_{C(\bH/\bK)}\circ\cJ\circ\cP_{\bC^F}\circ(\idty\otimes\cF\otimes\idty).
	\end{equation}
	In particular, the Fourier-Bessel operator is a bijection onto its range.
\end{theorem}

\begin{proof}
	It is clear that it suffices to prove the statement for a basis of $\bC^\bK\otimes\bC^\bK\otimes\bC^F$ as, for example, 
	\begin{equation}
		B = \left\{ \delta_k\otimes \hat n\otimes \varphi\mid  k\in\bK,\,  \hat n\in\widehat\bK,\,\varphi\in\bC^F\right\}.
	\end{equation}
	Observe that $\cF\hat n = \delta_{\hat n}$ and that $\cP_{\bC^F}(\delta_k\otimes\delta_{\hat n}\otimes\varphi) = \delta_k\delta_{\hat n}\varphi$.
	Thus, 
	\begin{equation}
		[\cJ\otimes\cP_{\bC^F}\circ(\idty\otimes\cF\otimes\idty).(\delta_k\otimes\hat n\otimes\varphi)]_{h,\hat m} = 
		\begin{cases}
			\overline{\hat n(k)}\, \cJ_{\hat n}\varphi & \text{ if } k=h,\, \hat n=\hat m,\\
			0& \text{ otherwise}.
		\end{cases}
	\end{equation}
	Then, considering the inverse actions $\cF^*$ and $\cP^{-1}_{C(\bH/\bK)}$, we have
	\begin{equation}\label{eq:part1}
		(\idty \otimes\cF^*\otimes\idty)\circ \cP^{-1}_{C(\bH/\bK)}\circ\cJ\otimes\cP_{\bC^F}\circ(\idty\otimes\cF\otimes\idty).(\delta_k\otimes\hat n\otimes\varphi) = \delta_k\otimes\hat n\otimes (\overline{\hat n(k)}\cJ_{\hat n}\varphi).
	\end{equation}

	Let us compute, by Proposition~\ref{prop:ap-simpl},
	\begin{equation}\label{eq:part2}
		\begin{split}
			\Xi\circ\cF_{\AP}^{-1}[\delta_k\otimes \hat n\otimes \varphi] (h,r,y)
			&=\cF_{\AP}^{-1}[\delta_k\otimes \hat n\otimes \varphi] (h,\phi(r)\sigma(y))\\
			&= \sum_{\lambda\in F} \sum_{\ell\in\bK} \phi(\ell h r^{-1})\lambda(\sigma(y))\delta_k(h) \hat n(\ell) \varphi(\lambda)\\
			&= \delta_k(h) \overline{\hat n}(h) \hat n(r) \sum_{\lambda\in F} \sum_{s\in\bK} \phi(s^{-1})\lambda(\sigma(y)) \overline{\hat n}(s) \varphi(\lambda)\\
			&=  [\delta_k\otimes \hat n \otimes (\overline{\hat n(k)}\cJ_{\hat n} \varphi)](h,r,y),
		\end{split}
	\end{equation}
	where we applied the change of variables $s = \ell h r^{-1}$. Together with \eqref{eq:part1}, this completes the proof.
\end{proof}

Via the lift procedure described in the previous section, the above yields a similar result on $\AP_F(\bH)$.

\begin{corollary}\label{cor:general-H}
	Let us consider the restriction of the Fourier-Bessel operator given by 
	\begin{equation}
		\cJ_\bH: \bigoplus_{\hat n\in\widehat\bK}\bC^F\to \bigoplus_{\hat n\in\widehat\bK}C(\bH/\bK),
		\qquad
		\cJ_\bH = \bigoplus_{\hat n\in\widehat \bK}\cJ_{\hat n}.
	\end{equation}
	Then, the bijection $\cF_{\AP}^{-1}:\bC^\bK\otimes\bC^F\to \AP_F(\bH)$ admits the following decomposition
	\begin{equation}
		\cF_{\AP}^{-1} = \Xi^{-1}\circ(\cF^*\otimes\idty)\circ \cP^{-1}_{C(\bH/\bK)}\circ\cJ_{\bH}\circ\cP_{\bC^F}\circ(\cF\otimes\idty),
	\end{equation}
	where $\cP_{\bC^F}:\bC^{\widehat\bK}\otimes \bC^F\to \bigoplus_{\hat n\in\widehat\bK}\bC^F$ and $\cP_{C(\bH/\bK)}:\bC^{\widehat\bK}\otimes C(\bH/\bK) \to \bigoplus_{\hat n\in\widehat \bK}C(\bH/\bK)$ are the appropriate restrictions of the corresponding operators given by \eqref{eq:P}.
\end{corollary}

\begin{proof}
	It suffices to check the statement on the basis of $\bC^\bK\otimes\bC^F$ given by $\{\hat n\otimes \varphi\mid\hat n\in\widehat\bK,\,\varphi\in\bC^F\}$. Then, if $\hat\psi = \hat n\otimes\varphi$ corresponds to $\psi=\cF_{\AP}^{-1}\hat\psi\in \AP_F(\bH)$, and letting $\Psi\in\AP_F(\bG)$ be the lift of $\psi$, we have that $\hat\Psi = \delta_0\otimes\hat n\otimes\varphi$.
	Then, the statement follows by \eqref{eq:part1} and \eqref{eq:part2}.
\end{proof}

\begin{rmk}\label{rmk:4}
	If a different lift from $\AP_F(\bH)$ to $\AP_F(\bG)$ is considered, the above corollary cannot be recovered. 
	This is easy to check, e.g., for the (left-invariant) lift $\Psi(k,\phi(k)x)=\psi(x)$. Indeed, in this case, if $\hat\psi = \hat n\otimes\varphi$ we have $\hat\Psi = \frac1N\sum_{k\in\bK} \delta_k\otimes\hat n\otimes\varphi$. Thus, by \eqref{eq:part2}
	\begin{equation}
		\Xi\circ\cF_{\AP}^{-1}\hat\psi(r,y) = \sum_{h\in\bK} \Xi\circ\ev\hat\Psi(h,hr,y) = \sum_{h\in\bK}\frac1N\sum_{k\in\bK} \delta_k(h)\,\hat n(hr)\, \overline{\hat n(k)}\cJ_{\hat n}\varphi(y) = [\hat n\otimes \cJ_{\hat n}\varphi](r,y).
	\end{equation}
	However, by \eqref{eq:part1},
	\begin{equation}
		\begin{split}
		(\cF^*\otimes\idty)\circ \cP^{-1}_{C(\bH/\bK)} & \circ\cJ_{\bH}\circ\cP_{\bC^F}  \circ(\cF\otimes\idty)\hat\psi(r,y)\\
		&=  \frac1N \sum_{h,k\in\bK}(\idty \otimes\cF^*\otimes\idty)\circ \cP^{-1}_{C(\bH/\bK)}\circ\cJ\circ\cP_{\bC^F}\circ(\idty\otimes\cF\otimes\idty).(\delta_k\otimes\hat n\otimes\varphi)(h,r,y) \\
		&= \frac1N\sum_{h,k\in\bK}[\delta_k\otimes\hat n\otimes (\overline{\hat n(k)}\cJ_{\hat n}\varphi)](h,r,y)\\
		&= \frac1N\left(\sum_{k\in\bK}\hat n(k)\right) [\hat n\otimes\cJ_{\hat n}\varphi](r,y).
		\end{split}
	\end{equation}
	Since $\frac1N\left(\sum_{k\in\bK}\hat n(k)\right) = \delta_0(\hat n)$, the above proves Corollay~\ref{cor:general-H} for functions of $\bC^\bK\otimes\bC^F$ independent on the first variable only. 

	The same reasoning shows, as announced in Remark~\ref{rmk:2}, that the approach used above cannot be extended to the case $\bG=SE(2)$, where $\bK$ is non-discrete.
\end{rmk}

\subsection{Almost periodic interpolation}\label{sec:ap-interp}

In this section we apply (and slightly generalize) the results of the previous section to the problem of interpolating and approximating functions between two fixed grids in $\cS$ and $\bG$, respectively.
In particular, we are interested in finite sets $\tilde E\subset \bG$ that are invariant under the action of $\bK$ both on $\bG$ and on the $\bH$ component of $\bG$.
These sets are completely determined by finite sets $E\subset \bH/\bK$ in the following way:
\begin{equation}
	g\in \tilde E \iff \exists y\in E, h,k\in\bK \text{ s.t.\ } g = (k,\phi(h)\sigma(y)),
\end{equation}
where $\sigma:\bH/\bK\to\bH$ is a fixed section.
This identification allows to decompose $\bC^{\tilde E}\simeq\bC^\bK\otimes\bC^\bK\otimes\bC^E$.
Then, we let the \emph{sampling operator} $\sampl:\bC^\bG\to \bC^\bK\otimes\bC^\bK\otimes\bC^E$ to be
\begin{equation}
	\sampl \psi(k,h,y) = \psi(k,\phi(h)\sigma(y)).
\end{equation}
Finally, the \emph{evaluation operator} $\ev:\bC^\bK\otimes\bC^\bK\otimes\bC^F\to \bC^\bK\otimes\bC^\bK\otimes\bC^E$ is defined as $\ev = \sampl\circ\cF_{\AP}^{-1}$. That is, $\ev$ is the operator associating to each $\hat f$ the sampling on $\tilde E$ of the corresponding $\AP_F(\bG)$ function.

\begin{definition}
	Let $F\subset \widehat\cS$ and $E\subset \bH/\bK$ be two finite sets. The \emph{almost-periodic (AP) interpolation} of a function $\Psi:\bG\to \bC$ on the couple $(E,F)$ is the function $f\in \AP_F(\bG)$ such that $\ev \hat f = \sampl\Psi$.
	We say that the AP interpolation problem on $(E,F)$ is \emph{well-posed} if to each $\Psi:\bG\to \bC$ corresponds exactly one AP interpolation $f\in\AP_F(\bG)$.
\end{definition}

In practice, even if the AP interpolation problem is well-posed, one has to pay some attention. Indeed, the AP interpolation $f\in\AP_F(\bH)$ of $\psi:\bH\to\bC$ can oscillate wildly in between points of $E$. This can be observed in Section~\ref{sec:numerical}, where it shown that this function behaves very badly w.r.t.\ small translations in space. (To this effect, see Figure~\ref{fig:ap-approx}.) Thus, we introduce also the following weighted version of the AP interpolation problem.

\begin{definition}
	Fix a vector $d\in\bR^\bK\otimes \bR^\bK\otimes \bR^F$.
	The \emph{AP approximation} of a function $\psi:\bG\to \bC$ on the couple $(E,F)$ is the function $f\in\AP_F(\bG)$ such that $\hat f\in\bC^\bK\otimes \bC^\bK\otimes \bC^F$ satisfies
	\begin{equation}\label{eq:ap-approx}
		\hat f =  \arg\min_{v\in\bC^\bK\otimes \bC^\bK\otimes\bC^F} \langle d, v\rangle + \| \sampl \psi - \ev v \|^2.
	\end{equation}
\end{definition}

It is clear that, if the AP interpolation problem problem is well-posed, the AP interpolation coincides with the AP interpolation with $d=0$.

To apply the results of the previous section to this setting, let us introduce the discretization of the generalized Fourier-Bessel operator.

\begin{definition}\label{def:discr-bessel}
	The \emph{discrete Fourier-Bessel operator on the couple $(E,F)$} is the operator 
	\begin{equation}
		\cJ^{E}:\bigoplus_{k\in\bK,\hat n\in \widehat \bK}\bC^F\to \bigoplus_{k\in\bK,\hat n\in \widehat \bK} \bC^E,
		\qquad
		\cJ^E =\left[ \bigoplus_{k\in\bK,\hat n\in\widehat\bK}\Pi \right] \circ\cJ,
	\end{equation}
	where $\Pi:C(\bH/\bK)\to\bC^E$ is the sampling operator $\Pi\varphi = (\varphi(y))_{y\in E}$.
%
\end{definition}

The following is the main result of the paper.

\begin{theorem}\label{thm:discrete}
	The operator $\ev:\bC^\bK\otimes\bC^\bK\otimes\bC^F\to \bC^\bK\otimes\bC^\bK\otimes\bC^E$ decomposes as follows.
	\begin{equation}\label{eq:ev}
		\ev = (\idty \otimes\cF^*\otimes\idty)\circ \cP^{-1}_{\bC^E}\circ\cJ^E\circ\cP_{\bC^F}\circ(\idty\otimes\cF\otimes\idty).
	\end{equation}
	In particular, the AP interpolation problem on $(E,F)$ is well-posed if and only if $\cJ^E$ is invertible.
\end{theorem}

\begin{proof}
	It follows directly from the respective definitions that
	\begin{equation}
		\sampl\circ\,\Xi^{-1} = \idty\otimes \idty\otimes \Pi.
	\end{equation}
	Thus, by definition of $\ev$, of $\cJ^E$, and Theorem~\ref{thm:general}, the statement is equivalent to
	\begin{equation}
		(\idty\otimes\idty\otimes\Pi)\circ(\idty\otimes\cF^*\otimes \idty)\circ \cP_{C(\bH/\bK)}^{-1} = 
		(\idty \otimes\cF^*\otimes\idty)\circ \cP^{-1}_{\bC^E}\circ\bigoplus_{k\in\bK,\hat n\in\widehat\bK}\Pi.
	\end{equation}
	Since, up to changing the identity operators, $(\idty\otimes\idty\otimes\Pi)$ and $(\idty\otimes\cF^*\otimes \idty)$ commute, this reduces to
	\begin{equation}
		(\idty\otimes\idty\otimes\Pi)\circ \cP_{C(\bH/\bK)}^{-1} = 
		 \cP^{-1}_{\bC^E}\circ\bigoplus_{k\in\bK,\hat n\in\widehat\bK}\Pi.
	\end{equation}
	Finally, the above holds, as can be easily seen by testing it on functions of the type $(\delta_h(k) \delta_{\hat m}(\hat n)\varphi)_{k\in\bK,\hat m\in\widehat\bK}$, for $h\in\bK$, $\hat n\in\widehat\bK$, and $\varphi\in C(\bH/\bK)$.
\end{proof}

It is clear that restricting the evaluation and sampling operators on vectors of the form $\delta_0\otimes v\otimes w$ allows to define the AP interpolation and approximation of functions $\psi:\bH\to\bC$.
In particular, the same arguments used in Corollary~\ref{cor:ap-simpl-H}, allow to prove the following.

\begin{corollary}\label{cor:discrete-H}
	Let us consider the restriction of the discrete Fourier-Bessel operator given by 
	\begin{equation}
		\cJ_\bH^E: \bigoplus_{\hat n\in\widehat\bK}\bC^F\to \bigoplus_{\hat n\in\widehat\bK}\bC^E,
		\qquad
		\cJ^E_\bH =\left[ \bigoplus_{\hat n\in\widehat\bK}\Pi \right]\circ \cJ_\bH.
	\end{equation}
	Then, $\ev:\bC^\bK\otimes\bC^F\to \bC^\bK\otimes\bC^E$ admits the following decomposition
	\begin{equation}
		\ev = (\cF^*\otimes\idty)\circ \cP^{-1}_{\bC^E}\circ\cJ_{\bH}^E\circ\cP_{\bC^F}\circ(\cF\otimes\idty),
	\end{equation}
	where $\cP_{\bC^F}:\bC^{\widehat\bK}\otimes \bC^F\to \bigoplus_{\hat n\in\widehat\bK}\bC^F$ and $\cP_{\bC^E}:\bC^{\widehat\bK}\otimes \bC^E \to \bigoplus_{\hat n\in\widehat \bK}\bC^E$ are the appropriate restrictions of the corresponding operators given by \eqref{eq:P}.

	In particular, the AP interpolation problem on $(E,F)$ is well-posed if and only if $\cJ_{\bH}^E$ is invertible.
\end{corollary}

\section{Application to image processing}\label{sec:image}

In this section we particularize the results of Section~\ref{sec:harmonic} to the almost-periodic interpolation of functions $\psi:\bR^2\to \bC$ on a spatial grid $\tilde E$ and a frequency grid $\tilde F$. These grids are assumed to be invariant under the action of $\bZ_N$ on $\bR^2$, given by the rotations $\{R_{\frac{2\pi}Nk}\}_{k\in\bZ_N}$. This is indeed a particular case of Corollary~\ref{cor:discrete-H}.

In this setting, we can naturally identify $\bR^2/\bZ_N$ with the slice $\cS_N=\{\rho e^{i\alpha}\mid \rho>0,\, \alpha \in[0,2\pi/N)\}$, thus fixing a choice for the section $\sigma$ and the map $\Xi$ introduced in Section~\ref{sec:harmonic}.
Clearly, the same is true for the set $\widehat\cS$ of frequencies with trivial stabilizer subgroup.
Since $\tilde E$ is rotationally invariant under discrete rotations in $\bZ_N$, we represent any element of $x\in \tilde E$ as a couple $(n,y)\in \bZ_N\times E$, where $E\subset \cS_N$, by letting
\begin{equation}
	x  = R_{\frac{2\pi}N n}y, \qquad y\in E, \, n\in\bZ_N.
\end{equation}
The same can be done for any $\Lambda\in\tilde F$, with $(n,\lambda)\in \bZ_N\times F$ and $F\subset \cS_N$.
Moreover, considering polar coordinater $y=\rho e^{i\alpha}$ and $\lambda=\xi e^{i\omega}$, letting $P,Q\in\bN$ be the respective cardinalities of $E$ and $F$, we will exploit the identifications 
\begin{equation}\label{eq:identification}
	E = \left\{(m,\rho_j e^{i\alpha_j})\mid n=0,\ldots, N,\, j=0,\ldots, P \right\},\quad
	F = \left\{(n,\xi_k e^{i\omega_k})\mid m=0,\ldots, N,\, k=0,\ldots, Q \right\}.
\end{equation}

As in the previous section, the sampling of a function $\varphi:\bR^2\to \bC$ on $\tilde E$ is given by the sampling operator, $\sampl: \varphi\mapsto\sampl\varphi \in \bC^N\otimes\bC^E$. On the other hand, the evaluation operator $\ev:\bC^N\otimes\bC^F\to \bC^N\otimes\bC^E$ associates to $\hat f\in\bC^N\otimes\bC^E$ the sampling on $\tilde E$ of the function $f:\bR^2\to \bC$ of the form \eqref{eq:almost-per-func}. That is,
\begin{equation} \label{eq:interpol}
	f(x) = \sum_{\lambda\in \tilde F} \hat f(\lambda) e^{i\langle \lambda, x\rangle}, \qquad \forall x\in \tilde E.
\end{equation}

Recall that $\bZ_N\simeq\widehat{\bZ_N}$ and $\bR^2\simeq\widehat{\bR^2}$. In the following we will let $\hat n(k) = e^{i\frac{2\pi}N \hat n k}$ and $\lambda(x) = e^{i\langle\lambda, x\rangle}$. In particular, in polar coordinates the latter becomes
\begin{equation}
	\lambda(x) = e^{i\lambda\rho\cos(\alpha-\omega)},\qquad \text{if }x = \rho e^{i\alpha} \text{ and } \lambda=\xi e^{i\omega}.
\end{equation}
Then, direct computations yield.

\begin{prop}
	The generalized Bessel function on $SE(2,N)=\bZ_N\ltimes\bR^2$ of parameter $\hat n\in\widehat\bZ_N$ is
	\begin{equation}
		J_{\hat n}(\lambda,y) = \sum_{r= 0}^{N-1} e^{i\xi\rho \cos\left(\alpha-\omega+\frac{2\pi}Nr\right)-\frac{2\pi}N\hat n r}, \qquad \text{if }x = \rho e^{i\alpha}\in \cS_N \text{ and } \lambda=\xi e^{i\omega}\in\cS_N.
	\end{equation}
	Moreover, the restriction of the discrete Fourier-Bessel operator $\cJ^E_{\bR^2}$ is the block-diagonal operator $\cJ^E_{\bR^2} = \bigoplus_{\hat n=0}^{N-1}\cJ_{\hat n}$, where $\cJ_{\hat n}: \bC^F\to \bC^E$ is given by the matrix,
	\begin{equation}
		 (\cJ_{\hat n})_{k,j} = J_{\hat n}(\xi_k e^{i\omega_k},\rho_j e^{i\alpha_j}).
	\end{equation}
\end{prop}

\begin{rmk}
	Generalized Bessel functions on $SE(2,N)$ only depend on the product $\xi\rho$ and on the difference $\alpha-\omega$. 
	Since $\alpha,\omega\in[0,2\pi/N)$, it is clear that, for $N\rightarrow+\infty$, the generalized Bessel functions converge to the usual ones:
	\begin{equation}
		J_{\hat n}(\xi,\rho) \xrightarrow{N\to+\infty} 2\pi i^{\hat n} J_{\hat n}(\xi\rho).
	\end{equation}
\end{rmk} 

As a consequence of the above result and Corollary~\ref{cor:discrete-H}, we have the following.
\begin{corollary}\label{cor:ap-r}
The sampling of $f\in\AP_F(\bR^2)$ is connected with $\hat f$ by
\begin{equation}
\big((\cF\otimes\idty)\sampl f\big)_{\hat n,j} = \sum_{k=0}^{Q-1} J_{\hat n}(\lambda_k,y_j)\big((\cF\otimes\idty)\hat f\big)_{\hat n, k}.
\end{equation}
In particular, the AP interpolation problem on $(E,F)$ is well-posed if and only if all the matrices $\cJ_{\hat n}^E$ are invertible. 
\end{corollary}

\begin{prop}\label{prop:least-squares}
	Let $P\ge Q$. Then, for a given weight vector $d\in\bR^N\otimes\bR^E$, the AP approximation of a function $\psi:\bR^2\to \bC$ on the couple $(E,F)$ is the function $f\in\AP_F(\bR^2)$ such that $\hat f = (\cF^*\otimes\idty) w$, where
	\begin{equation}
		\left(\cJ_{\hat n}^*\circ\cJ_{\hat n}- \diag_{i} d_{\hat n,i}^2\right) w_{\hat n,\cdot}  = \cJ^*_{\hat n} [(\cF\otimes \idty)\sampl\psi]_{\hat n,\cdot}
		\quad\text{for any}\quad \hat n=0,\ldots,N-1.
	\end{equation}
\end{prop}

\begin{proof}
	From Corollary~\ref{cor:ap-r}, the definition of AP approximation, and the fact that $\cF\otimes\idty$ is an isometry, we have that
	\begin{equation}\label{eq:ap-approx}
		(\cF\otimes\idty) \hat f 
			=  \arg\min_{v\in\bC^N\otimes\bC^F} \langle d, v\rangle + \| \cP_{\bC^E}\circ(\cF\otimes \idty) \sampl \psi - \cJ \circ \cP_{\bC^F} v \|^2.
	\end{equation}
	In particular, this decomposes for $\hat n\in\{0,\ldots,N-1\}$ as
	\begin{equation}
		(\cF\otimes\idty) \hat f_{\hat n,\cdot} = \arg\min_{v_{\hat n}\in\bC^F} \langle d_{\hat n,\cdot},v_{\hat n}\rangle + \|[(\cF\otimes\idty)\sampl\psi]_{\hat n,\cdot} - \cJ_{\hat n} v_{\hat n}\|^2.
	\end{equation}
	The statement then follows by the standard formula for solving complex least-square problems.
\end{proof}

\begin{rmk}
	In numerical experiments, we always found the matrix conditioning of the matrices $\cJ_{\hat n}$ to be very good.
	Moreover, these same experiment seem to suggest this conditioning to be connected with the smallest distance between elements in $E$ and in $F$. 
\end{rmk}

\subsection{Computational cost}

\begin{prop}\label{prop:comp-cost}
	Given $\sampl\psi$, after a prefactorization of $\cJ$ of computational cost $\cO(N Q^3)$, the computational cost of the AP approximation is 
	\begin{equation}\label{eq:comp-cost}
		\cC = N\left(\frac{Q^2}2 + 10\max\{P,Q\}\log N\right).
	\end{equation}
	Moreover, this operation can be parallelized on $N$ processors, yielding an effective cost of
	\begin{equation}\label{eq:comp-cost-parallel}
		\cC_P = \frac{Q^2}2 + 10\max\{P,Q\} N\log N.
	\end{equation}
\end{prop}

\begin{proof}
	Let $v\in\bC^N\otimes \bC^E$ and denote $\tilde v_{\hat n} = (\cF\otimes\idty)v_{\hat n, \cdot}$. Similarly, for $w=\ev(v)$ let $\tilde w_{\hat n,\cdot} = (\cF\otimes\idty)w_{\hat n,\cdot}$.  The computational cost to evaluate $(\cF\otimes\idty)v$ and to pass from the $\tilde w_{\hat n}$'s to $w$ is of $5 P N\log N$ and $5 QN\log N$ FLOPs, respectively. By Proposition~\ref{prop:least-squares}, solving \eqref{eq:ap-approx} amounts to solve, the following problems
	\begin{equation}
		(\cJ_{\hat n}^*\circ\cJ_{\hat n}-  \diag_{i}d_{\hat n,i}^2) \tilde v_{\hat n} = \cJ^*_{\hat n} \tilde w_{\hat n} \qquad \hat n = 0,\ldots,N-1.
	\end{equation}
	Up to a prefactorization of the matrices $\cJ_n^*\circ\cJ_n-\diag_{i}d_{\hat n,i}^2$, with a computational cost of $\cO(Q^3)$, solving each of the above systems has a computational cost of $Q^2/2$ FLOPs. All together this yields the (non-parallelized) final cost of \eqref{eq:comp-cost}
	Since the solution of the systems is independent for each $n$, it can be parallelized, yielding to the cost \eqref{eq:comp-cost-parallel}, for $N$ processors.
\end{proof}

\begin{corollary}
	Let $G = \{ \rho_j e^{i \frac{2\pi}K k}\mid j=1,\ldots,R ,\, k = 1,\ldots,K \}$ be a fixed polar grid. Then, for $\tilde E=\tilde F=G$, the best choice for AP approximation is $N=\sqrt{{|G|}/{10}}$. This yields a prefactorization complexity of $\cO(|G|^{3/2})$ and a computational complexity of
	\begin{equation}\label{eq:polar-comp-cost}
		\cC =  \cO(|G|^{3/2}\log |G|)
		\quad\text{and}\quad
		\cC_P = 5|G|\left( 1 + \log\frac{|G|}{10}\right).
	\end{equation}
\end{corollary}

\begin{proof}
	Clearly, $P=Q$. 	Then, a simple computation, using that $|G|=|E|=QN$, yields
	\begin{equation}
		\cC_P = \cC_P(M) = \frac{|G|^2}{2N^2} + 10 |G|\log N.
	\end{equation}
	The above expression attains its minimum at $N=\sqrt{|G|/10}$, which gives the cost in \eqref{eq:polar-comp-cost}. To complete the proof for $\cC$ it suffices to observe that $\cC = N\,\cC_P$.
\end{proof}

\begin{rmk}
	The above shows that, once parallelized, the complexity of the AP approximation is the same as the polar Fourier transform algorithm presented in \cite{bbbb}.
\end{rmk}


\subsection{Numerical tests}\label{sec:numerical}

The numerical implementation of the AP interpolation and approximation procedures has been developed in \verb+Julia+. The main program and the tests are contained in the package \verb+ApApproximation.jl+, which is available at \url{http://github.com/dprn/ApApproximation.jl}, and in particular in \href{http://nbviewer.jupyter.org/github/dprn/ApApproximation.jl/blob/master/notebooks/AP\%20Interpolation\%20and\%20approximation\%20tests.ipynb}{this Jupyter notebook}. 

For $\xi\in \mathbb R^2$ let $\tau_\xi$ be the translation $\tau_\xi f(x) = f(x-\xi)$. Then, if $f$ is of the form \eqref{eq:almost-per-func}, the same is true for $\tau_\xi f$, with
\begin{equation}
	\widehat{\tau_\xi f}_{k,m} = e^{-i\big\langle R_{\frac{2\pi m}N}\Lambda_k, \xi \big\rangle}\hat f_{k,m}.
\end{equation}
Let $\hat\tau_\xi$ be defined by $\hat\tau_\xi(\hat f):=\widehat{\tau_\xi f}$. In our tests we exploited this operator to check the results of the AP interpolation and approximation. Indeed, in general, applying a translation will completely change the points on which $f$ is sampled by $\ev$ and hence highlights the presence of high variability in between the points of interpolation/approximation.

For the tests, we fixed $Q=340$, $N=64$, and defined a specific set $E=F$. We then computed the AP interpolation, resp. approximation, with respect to these sets. As weights for the AP approximation we chose
\begin{equation}\label{eq:weights}
    d(\Lambda) = 
    \begin{cases}
    \frac \alpha{10} &\qquad \text{if } |\Lambda|\le 1\\
    \alpha &\qquad \text{if } 1<|\Lambda|\le\frac32\\
    100\alpha &\qquad \text{if } |\Lambda|>\frac32.\\
    \end{cases}
\end{equation}
In Figure~\ref{fig:ap-approx} we present, from left to right, a plot of the
power spectrum of $\hat f$ and the results of the evaluation of $\ev\hat f$,
of $\ev(R_\gamma\hat f)$ for the angle $\gamma = 10\times 2\pi/N$,  and of
$\ev(\hat\tau_\xi\hat f)$ for $\xi\sim (15,26)$. In Table~\ref{tab:norms}, we present the corresponding $L^2$ norms. In particular, we observe that the $L^2$ norm of $\ev(\hat\tau_\xi\hat f)$ is stable for the AP approximation, contrarily to what happens for the AP interpolation.
Obviously, we see also that the effect of discrete rotations is perfect for both interpolation and approximation.

\begin{figure}
	\begin{minipage}[b]{.24\linewidth}
		\centering\includegraphics[width=.7\textwidth]{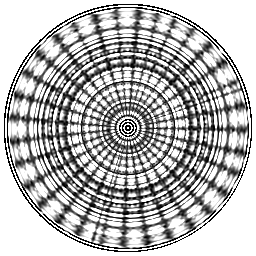}
	\end{minipage}
	\begin{minipage}[b]{.24\linewidth}
		\centering\includegraphics[width=.7\textwidth]{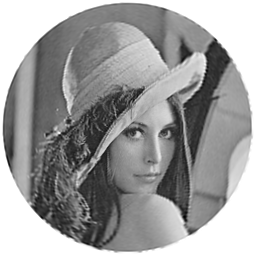}
	\end{minipage}%
  \begin{minipage}[b]{.24\linewidth}
		\centering\includegraphics[width=.7\textwidth]{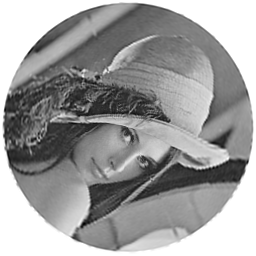}
	\end{minipage}
  \begin{minipage}[b]{.24\linewidth}
		\centering\includegraphics[width=.7\textwidth]{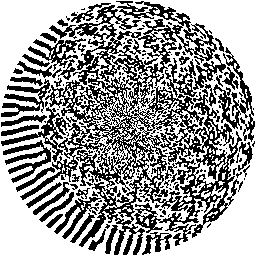}
	\end{minipage}
	\begin{minipage}[b]{.24\linewidth}
		\centering\includegraphics[width=.7\textwidth]{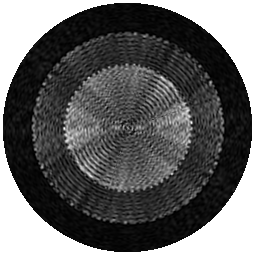}
		\subcaption{Magnitudes: $|\hat f|^2$.\\ $ $\\ $ $}
		\label{fig:1b}
	\end{minipage}
	\begin{minipage}[b]{.24\linewidth}
		\centering\includegraphics[width=.7\textwidth]{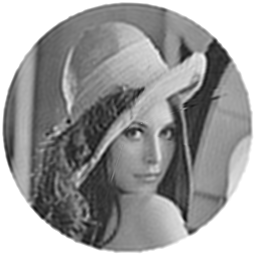}
		\subcaption{Evaluation: $\ev \,\hat f$.\\ $ $\\ $ $}
		\label{fig:1a}
	\end{minipage}%
	\begin{minipage}[b]{.24\linewidth}
		\centering\includegraphics[width=.7\textwidth]{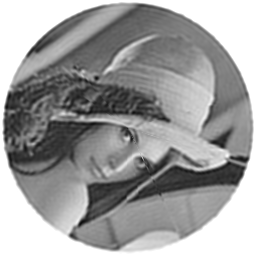}
		\subcaption{Rotation of $\gamma = 20\pi/N$ 
      and evaluation: $\ev \big(R_\gamma\hat f\big)$.\\$ $}
		\label{fig:1b}
	\end{minipage}
  \begin{minipage}[b]{.24\linewidth}
		\centering\includegraphics[width=.7\textwidth]{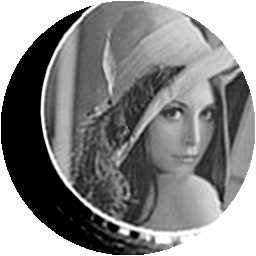}
		\subcaption{Translation of $\xi\sim(15,26)$ and
      evaluation: $\ev\big(\hat\tau_{\xi}\hat f\big)$.}
		\label{fig:1b}
	\end{minipage}
	\caption{Tests on a $246\times256$ image with $E=F$, $Q = 340$, $N=64$. \emph{First row:} AP interpolation. \emph{Second row:} AP approximation with weights \eqref{eq:weights} and $\alpha = 100$.}
	\label{fig:ap-approx}
\end{figure}

\begin{table}
	\centering
	\begin{tabular}{l || c | c | c | c |}
		 & $\hat f$ & $\ev\big(\hat f\big)$ & $\ev \big(R_\gamma\hat f\big)$& $\ev\big(\hat\tau_{\xi}\hat f\big)$ \\
		 \hline
		AP interpolation & $6.7\times 10^{5}$ & $80.3$ & $80.3$ & $3.0 \times 10^{10}$\\
		AP approximation & $0.2$ & $80.0$ & $80.0$ &  $79.0$ \\
	\end{tabular}	
	\caption{$L^2$ norms of the AP interpolation and approximation of
    Figure~\ref{fig:ap-approx} of an image with $L^2$ norm $80.1$ on the set
    $E$. In the first column we have the $L^2$ norms of the vector of
    frequencies $\{\hat f(\Lambda)\}_{\Lambda\in F}$, while in the second, third
    and last ones we have the $L^2$ norms of the vector obtained by applying the
    evaluation operator to $\hat f$, to its rotation by $\gamma =
    10\times2\pi/N$, and to its translation by $\xi\tilde (15,26)$, respectively.}
	\label{tab:norms}
\end{table}

\appendix

\section{Irreducible unitary representations of semidirect products}\label{sec:aux}

In this section we prove that the $T^\lambda$'s introduced in Section~\ref{sec:harmonic} are indeed representations of $\bG$. Although this is a trivial consequence of Mackey's theory, we preferred to present here a direct proof of this fact.

\begin{lemma}\label{lem:irrep}
	For any $\lambda\in \cS\subset\widehat\bH$ the operator on $L^2(\bK)\simeq \bC^N$ given by
	\begin{equation}
		T^\lambda(k,x) = \diag_h\left( \phi(h)\lambda(x) \right)S(k),
	\end{equation}
	is a unitary irreducible representations
\end{lemma}

\begin{proof}
	Observe that, for any vector $v\in \bC^N$ and any $k\in\bK$ we have $S(k)\circ\diag v = \diag(S(k)v)\circ S(k)$. Thus, for any $(k,x),(h,y)\in \bG$, 
	\begin{multline}
			T^\lambda(k,x)\circ T^\lambda(h,y) 
			= \diag_r\left(\phi(r)\lambda(x)\right) \circ \diag_r \left( \phi(k^{-1}r)\lambda(y) \right)\circ S(hk)\\
			= \diag_r\left(\phi(r)\lambda(x+\phi(k)y)\right)\circ S(hk)
			= T^\lambda\big((k,x)(h,y)\big).
	\end{multline}
	This proves that $T^\lambda$ is a representation.
	
	To prove that $T^\lambda$ is unitary, observe that $S(k)^* = S(k^{-1})$ for all $k\in \bK$ and so that
	\begin{multline}
		T^\lambda(k,x)^{-1}
		= T^\lambda\big((k,x)^{-1}\big) 
		= T^\lambda(k^{-1},-\phi(k^{-1})x) 
		= \diag_r\left( \phi(kr)\bar\lambda(x)  \right)\circ S(k^{-1})\\
		= S(k^{-1})\circ \diag_r\left( \phi(r)\bar\lambda(x)  \right)
		= T^\lambda(k,x)^*.
	\end{multline}
	
	The irreducibility can be directly proved by Schur's Lemma for unitary representations \cite[Proposition 4]{Barut77theoryof}. Namely, it suffices to show that the only $N\times N$ matrices commuting with all the $T^\lambda(k,x)$ are the scalar multiples of the identity.
		Let $M\in\bC^{N\times N}$ be commuting with all the $T^\lambda(k,x)$. Then, in particular, it commutes with $T^\lambda(k,0) = S(k)$ for all $k\in\bK$. A direct computation shows that this implies $M_{r,s} = M_{rk^{-1},sk^{-1}}$ for all $r,s,k\in\bK$, i.e., that the matrix is circulant.
		On the other hand, $M$ has to commute with $T^\lambda(e,x) = \diag_h\left( \phi(h)\lambda(x) \right)$, where $e$ is the identity element in $\bK$.
		Again, direct computations yield that $M_{r,s} = \lambda(\phi(r^{-1})x-\phi(s^{-1})x) M_{r,s}$ for all $r,s\in\bK$, that is, $M$ is diagonal.
		Since the only circulant and diagonal matrices are the scalar multiples of the identity, this completes the proof.
\end{proof}

\section*{Aknowledgements} This research has been supported by the Grant
ANR-15-CE40-0018 of the ANR, and partially supported by the European Research
Council, ERC StG 2009 “GeCoMethods”, contract n. 239748. 

\section*{References}

\bibliographystyle{elsarticle-num}

\end{document}